\documentclass[11pt,reqno]{amsart}
\usepackage{amssymb,mathtools,calc,verbatim,enumitem,tikz,url,hyperref,mathrsfs,cite,fullpage}
\usepackage{bbm}
\usepackage{stmaryrd}
\usepackage{textcomp}
\usepackage{setspace}
\usepackage{amsthm}
\usepackage{amsmath}
\usepackage{graphicx}
\usepackage{marvosym}
\usepackage{empheq}
\usepackage{latexsym}
\usepackage[T1]{fontenc}
\usepackage{color}
\usepackage{cleveref}
\usepackage{dsfont}

\usepackage{todonotes}

\newenvironment{poc}{\begin{proof}[Proof of claim]}{\end{proof}}

\newtheorem{theorem}{Theorem}[section]
\newtheorem{lemma}[theorem]{Lemma}

\newtheorem{proposition}[theorem]{Proposition}
\newtheorem{claim}[theorem]{Claim}
\newtheorem*{claim*}{Claim}

\theoremstyle{definition}

\newtheorem*{qu*}{Question}
\theoremstyle{remark}

\renewcommand\Pr{\operatorname{\mathbb{P}}}

\renewcommand\geq{\geqslant}
\renewcommand\le{\leqslant}
\renewcommand\ge{\geqslant}
\renewcommand\to{\rightarrow}

\begin{document}

\title{Tight Staircase Bounds for Cyclic Subsets below Dirac's Threshold}

\author{Hong Liu}
\address{ECOPRO, Institute for Basic Science, 55 Expo-ro, Yuseong-gu, Daejeon, 34126, Korea}
\email{hongliu@ibs.re.kr}

\author{Mengyuan Niu}
\address{School of Mathematics and Statistics, Zhengzhou University, Zhengzhou, China, and ECOPRO, Institute for Basic Science, 55 Expo-ro, Yuseong-gu, Daejeon, 34126, Korea}
\email{mengyuanniu@gs.zzu.edu.cn}

\author{Lanchao Wang}
\address{School of Mathematics, Nanjing University, Nanjing, China, and ECOPRO, Institute for Basic Science, 55 Expo-ro, Yuseong-gu, Daejeon, 34126, Korea}
\email{lanchaowang@foxmail.com}

\author{Zhifei Yan}
\address{ECOPRO, Institute for Basic Science, 55 Expo-ro, Yuseong-gu, Daejeon, 34126, Korea}
\email{zhifeiyan@ibs.re.kr}

\begin{abstract}
Let $\operatorname{Cyc}(G)$ denote the number of  cyclic subsets in a graph $G$, which are subsets that induce a Hamiltonian subgraph. Dragani\'{c}, Keevash and M\"{u}yesser  recently proved that every regular Dirac graph has $\Omega(2^n)$ cyclic subsets, resolving a  problem of Erd\H{o}s and Faudree.

We determine the sharp asymptotic lower bound  throughout the linear range below Dirac's threshold. 
Let $G$ be an $n$-vertex $d$-regular graph with $d=\Omega(n)$ and $d<n/2$, then
$$
        \operatorname{Cyc}(G)\ge (q-o(1))2^{n/q}, \quad \text{where } \quad q=\left\lfloor \frac{n}{d+1}\right\rfloor \ge 2.
$$
This bound is asymptotically best possible, including the leading coefficient $q$, as witnessed at the staircase levels by the disjoint union of $q$ equal cliques. Consequently, the optimal exponential rate changes by discrete jumps as $d$ crosses the thresholds $n/k$, rather than varying smoothly with $d$. We also prove the optimal exponential rate at the Dirac boundary: every $n$-vertex $n/2$-regular graph  satisfies        $\operatorname{Cyc}(G)\ge 2^{(1-o(1))n},$
which is sharp up to a subexponential factor by $K_{n/2,n/2}$. 
\end{abstract}

\maketitle

\section{Introduction}\label{sec:intro}

Hamiltonicity is a central example of how local degree conditions force global structure.  Dirac's theorem asserts that every \(n\)-vertex graph with minimum degree at least \(n/2\) is Hamiltonian, and much of extremal graph theory studies how far such a conclusion can be strengthened.  In this paper we study an enumerative strengthening: rather than asking only whether \(G\) is Hamiltonian, we ask how many induced subgraphs of \(G\) are Hamiltonian.
For a vertex subset \(S\subseteq V(G)\), we say that \(S\) is a \emph{cyclic subset} if the induced subgraph \(G[S]\) contains a Hamiltonian cycle.\footnote{Throughout this paper, to handle boundary cases uniformly, we also regard the empty set, one-vertex sets, and copies of \(K_2\) as cyclic. Thus every subset of a complete graph is cyclic.}  We write
$\operatorname{Cyc}(G)
        :=
        \bigl|\{S\subseteq V(G):S\text{ is cyclic}\}\bigr|.$

This viewpoint belongs to a broader line of work showing that Dirac-type degree conditions force Hamiltonicity in robust forms.  For example, Dirac graphs contain many Hamilton cycles~\cite{CK09,SSS03}, their Hamiltonicity is robust under random sparsification~\cite{KLS14}, and related random, resilience, packing, and decomposition problems have been studied extensively; see, for instance,~\cite{FK08,Montgomery19,KOT10,CKLOT16}.  Counting cyclic subsets can be viewed as a vertex-subset analogue of these robustness phenomena.  Related subset-counting problems have also been studied for other spanning structures, including \(K_r\)-factors in regular graphs~\cite{SWY25}, and for cyclic subsets in tournaments~\cite{HLMS26}.

A first benchmark is provided by the minimum-degree problem.  Koml\'{o}s conjectured in 1981 that every graph \(G\) with minimum degree \(d\) contains at least as many cyclic subsets as the complete graph \(K_{d+1}\), namely \(\operatorname{Cyc}(G)\ge 2^{d+1}\).  Kim, Liu, Sharifzadeh and Staden~\cite{KLSS17} proved this conjecture for all sufficiently large \(d\).  Thus, under a pure minimum-degree assumption, the guaranteed number of cyclic subsets is controlled by the local parameter \(d\).

A different phenomenon appears when the order of the graph $n$ is fixed and the degree is linear.  At  the Dirac scale, the local benchmark $2^d$ is exponentially smaller than the total number $2^n$ of vertex subsets, so it becomes natural to ask when one can force cyclic subsets on the scale of all subsets. Erd\H{o}s and Faudree~\cite{Erdos97} asked, in the regular setting, whether every $(m+1)$-regular graph on $2m$ vertices has  $\Omega(2^{2m})$ cyclic subsets.  Dragani\'{c}, Keevash and M\"{u}yesser~\cite{DKM25}  recently answered this question, proving a sharp positive-fraction result for regular graphs above the Dirac threshold.

Our main result determines what happens below that threshold.  The first obstruction is that a \(d\)-regular graph may have several components.  Since each component has at least \(d+1\) vertices, the number of components is at most
\[
        q=\left\lfloor\frac{n}{d+1}\right\rfloor.
\]
At the points \(d=n/q-1\), the disjoint union \(qK_{n/q}\) shows that one cannot force more than \(q\cdot 2^{n/q}\) cyclic subsets, up to lower-order terms.  This suggests that the correct exponential scale below the Dirac threshold is governed not directly by \(d\), but by the integer \(q\).  Theorem~\ref{thm:main} confirms this prediction throughout the whole linear below-Dirac range.  In particular, the optimal exponential rate changes in discrete steps: it remains \(n/q\) throughout the interval where \(q=\lfloor n/(d+1)\rfloor\), and then jumps when \(q\) changes.

\begin{theorem}\label{thm:main}
Given $\varepsilon,\xi > 0$, there exists $n_0$ such that the following holds for all $n \ge n_0$. Let $G$ be an $n$-vertex $d$-regular graph with $\varepsilon n\le d< n/2$, and let $q = \left\lfloor n/(d+1)\right\rfloor\ge 2$. Then
$$\operatorname{Cyc}(G)\geq (q-\xi)\,2^{n/q}.$$
\end{theorem}

Theorem~\ref{thm:main}  is sharp for every fixed below-Dirac step  $q\ge2$ as witnessed by the above union of cliques $qK_{n/q}$.
The exact boundary at $d=n/2$ behaves differently.  As  observed by Erd\H{o}s and Faudree~\cite{Erdos97}, the  balanced complete bipartite graph satisfies
$
\operatorname{Cyc}(K_{n/2,n/2})=\Theta\left(\frac{2^n}{\sqrt{n}}\right),
$
so  one cannot force a positive fraction of all subsets to be cyclic.  Nevertheless, the same methods give the correct exponential rate at the boundary.

\begin{proposition}\label{prop:dirac_rate}
For every $\xi>0$, there exists $n_0=n_0(\xi)$ such that the following holds for all $n\ge n_0$. Let $G$ be an $n$-vertex $n/2$-regular graph. Then
$$
\operatorname{Cyc}(G)\ge 2^{(1-\xi)n}.
$$
\end{proposition}

Together, Theorem~\ref{thm:main} and Proposition~\ref{prop:dirac_rate} describe the exponential rate below the Dirac threshold and at the boundary: below the threshold the rate follows a staircase, while at the boundary it reaches \(2^{(1-o(1))n}\); see Figure~\ref{fig:staircase}. The regularity assumption is essential for this staircase phenomenon.  Under a minimum-degree condition alone, the components need not be balanced, and the extremal obstruction is no longer governed by \(q\) equal pieces.  For example, in the range        $q(d+1)\le n<(q+1)(d+1),$
the graph
$        (q-1)K_{d+1}\cup K_{\,n-(q-1)(d+1)}$
has minimum degree \(d\), but its cyclic subsets are dominated by the last clique, whose order varies with \(d\) even while \(q\) remains fixed.

\begin{figure}[htbp]
\centering
\begin{tikzpicture}[scale=0.9]

\draw[->, thick] (-0.2, 0) -- (6.5, 0) node[right] {$d$};
\draw[->, thick] (0, -0.2) -- (0, 4.8) node[above, yshift=0.15cm] {$n/{\log_2 \operatorname{Cyc}(G)}$};

\draw (1.5, 0.05) -- (1.5, -0.05) node[below] {$\frac{n}{4}$};
\draw (2, 0.05) -- (2, -0.05) node[below] {$\frac{n}{3}$};
\draw (3, 0.05) -- (3, -0.05) node[below] {$\frac{n}{2}$};
\draw (6, 0.05) -- (6, -0.05) node[below] {$n$};

\foreach \y in {1, 2, 3, 4} {
    \draw (0.05, \y) -- (-0.05, \y) node[left] {$\y$};
}

\draw[thick, red] (1.2, 4) -- (1.45, 4);
\filldraw[red] (1.2, 4) circle (1.5pt);
\draw[red, fill=white, thick] (1.5, 4) circle (1.5pt);

\draw[thick, red] (1.5, 3) -- (1.95, 3);
\filldraw[red] (1.5, 3) circle (1.5pt);
\draw[red, fill=white, thick] (2, 3) circle (1.5pt);

\draw[thick, red] (2, 2) -- (2.95, 2);
\filldraw[red] (2, 2) circle (1.5pt);
\draw[red, fill=white, thick] (3, 2) circle (1.5pt);

\draw[thick, red] (3, 1) -- (6, 1);
\filldraw[red] (3, 1) circle (1.5pt);
\filldraw[red] (6, 1) circle (1.5pt);

\draw[dashed, gray!60] (1.5, 0) -- (1.5, 4);
\draw[dashed, gray!60] (2, 0) -- (2, 3);
\draw[dashed, gray!60] (3, 0) -- (3, 2);

\draw[dashed, gray!60] (0, 1) -- (3, 1);
\draw[dashed, gray!60] (0, 2) -- (3, 2);
\draw[dashed, gray!60] (0, 3) -- (2, 3);
\draw[dashed, gray!60] (0, 4) -- (1.5, 4);

\end{tikzpicture}
\caption{\footnotesize The staircase behaviour of the exponential lower bound. As the regular degree $d$ drops below the thresholds $n/k$ (moving from right to left), the value of $n/{\log_2 \operatorname{Cyc}(G)}$ jumps discretely to the next integer.}
\label{fig:staircase}
\end{figure}
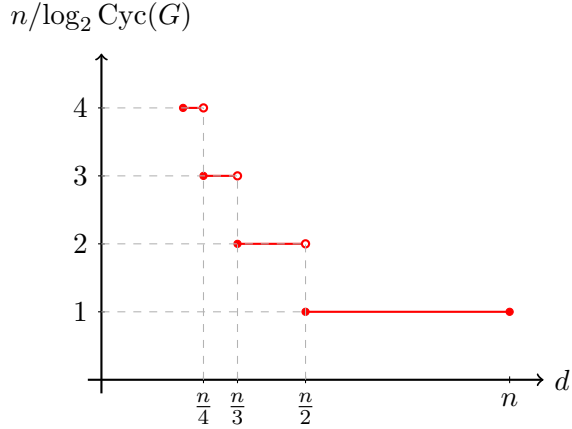

The proof of Theorem~\ref{thm:main} is not a direct extension of the argument of Dragani\'{c}, Keevash and M\"{u}yesser.  Above the Dirac threshold, regularity strongly restricts the global structure.  Below it, several macroscopic pieces may coexist, and the extremal obstruction is no longer a single Dirac-type configuration but a whole staircase of possible decompositions.  The main difficulty is therefore to distinguish the genuinely extremal \(q\)-block situation from configurations that only appear extremal at the reduced-graph level.

We do this through a trichotomy for the reduced graph.  Either there is a large connected matching, which can be converted through regular pairs into exponentially many cyclic subsets; or there are \(q\) balanced dense components, where one must count almost all subsets inside each component and preserve the sharp leading coefficient \(q\); or there are \(q+1\) near-critical components, where the regularity of \(G\) forces compensating edges between two blocks and yields an exponential surplus over the target bound.  The proof overview in Section~\ref{sec:overview} explains these three mechanisms in more detail.

The rest of the paper is organized as follows.  Section~\ref{sec:overview} gives a detailed overview of the proof.  Section~\ref{sec:dichotomy} proves the reduced-graph trichotomy.  Sections~\ref{sec:non_degenerate}, \ref{sec:balanced_dense}, and~\ref{sec:near_critical} treat the three alternatives: connected matchings, \(q\) balanced dense components, and \(q+1\) near-critical components.  Section~\ref{sec:proof} combines these ingredients to prove Theorem~\ref{thm:main}.  Section~\ref{sec:q1_sharpness} proves the boundary result, Proposition~\ref{prop:dirac_rate}, and Section~\ref{sec:conclusion} discusses several further directions.

\section{Proof overview}\label{sec:overview}

To prove Theorem~\ref{thm:main}, we first use regularity and the global regular degree condition to reduce the graph to one of a few possible large-scale shapes. Then in each shape we use a different counting mechanism to produce enough cyclic subsets. The staircase behaviour comes from the fact that, below the Dirac threshold, the reduced graph can split into $q$ pieces, and the extremal contribution is then obtained by counting subsets inside pieces of order about $n/q$.

\subsection*{The reduced trichotomy}

Apply the degree form of Szemer\'edi's Regularity Lemma to $G$, and let $R$ be the reduced graph. 
The regularity lemma transfers the regular degree condition of $G$ into the lower bound
$
\delta(R)\ge \left(\frac{1}{q+1}-o(1)\right)|R|.
$
This minimum degree condition severely restricts the component structure of $R$. A connected matching in $R$ is a matching contained in a single connected component of $R$.
Lemma~\ref{lem:trichotomy} shows that one of three alternatives must occur:
\begin{itemize}
\item[(a)] $R$ contains a connected matching covering  $(1/q+\Omega(1))|R|$ vertices;
\item[(b)] $R$ has exactly $q$ connected components, all of size about $|R|/q$;
\item[(c)] $R$ has exactly $q+1$ connected components, all of size about $|R|/(q+1)$.
\end{itemize}
The proof of this trichotomy is based on a longest-path argument. Let $C_1,\ldots,C_s$ be the components of $R$. The minimum-degree condition gives a lower bound roughly
$
|C_i|\ge |R|/({q+1})
$
for every component $C_i$, and hence $s\le q+1$. If $R$ contains a long path, then taking alternating edges on this path gives a large connected matching, which is Case (a). Otherwise, the longest path is short, we get an upper bound roughly
$
|C_i|\le |R|/{q}
$
for every component, so $s\ge q$. Thus only $s=q$ and $s=q+1$ remain, and a simple size calculation yields the balanced alternatives in Cases (b) and (c), respectively.

\subsection*{Case (a): a large connected matching in $R$}
Let
$
P_1Q_1,\ldots,P_sQ_s
$
be the connected matching in $R$. Each matching edge corresponds to a regular pair in the original graph. The connectedness of the matching allows us to link these pairs cyclically by short connecting paths $L_1,\ldots,L_s$, where $L_i$ connects $v_i\in Q_i$ to $u_{i+1}\in P_{i+1}$ for each $i\in [s]$, with indices taken modulo $s$.

Inside each super-regular pair $(P_i,Q_i)$, we prove an abundant Hamilton-path statement: for fixed endpoints $u_i\in P_i$ and $v_i\in Q_i$, at least $2^{|P_i|+|Q_i|-o(n)}$ balanced subpairs $(P_i',Q_i')$ support a Hamiltonian path from $u_i$ to $v_i$. Choosing such subpairs independently over all matching edges and concatenating the resulting Hamiltonian paths with the  connecting paths $L_1,...,L_s$ gives at least
  $$
\prod_{i\in[s]}2^{|P_i|+|Q_i|-o(n)}
=2^{\sum_{i\in[s]}(|P_i|+|Q_i|)-o(n)}
\ge 2^{2sn/|R|-o(n)}
\ge 2^{n/q+\Omega(n)}.
$$ 

\subsection*{Case (b): $q$ components of size about $|R|/q$}
This is precisely the extremal configuration for the staircase bound. Since each component has order about $n/q$, obtaining the lower bound $(q-o(1))2^{n/q}$, and in particular the sharp leading coefficient $q$, requires us to count cyclic subsets that lie almost entirely inside the individual dense components.

After passing back to  $G$, each component of $R$ gives rise to one vertex block. More precisely, if $C_i$ is a connected component of $R$, then $A_i$ denotes the union of the clusters of $G$ corresponding to the vertices of $C_i$. The remaining vertices come from the exceptional set in the regularity partition; we distribute these exceptional vertices among the blocks and denote by $Z_i$ the set assigned to $A_i$. Without incorporating the exceptional vertices into the blocks, this argument would only yield $2^{n/q-o(n)}$, rather than the sharp leading term $(q-o(1))2^{n/q}$. Thus each block has the form
$
B_i=A_i\cup Z_i.
$
The set $Z_i$ is small, while the minimum-degree condition inherited from $G$ ensures that $G[A_i]$ has minimum degree strictly larger than $|A_i|/2$, and that every vertex in $Z_i$ still has many neighbours in $A_i$.

A single-block counting lemma shows that almost every subset of such a block is cyclic. Indeed, if $S\subseteq B_i$ is chosen uniformly at random, then $S\cap A_i$ inherits a Dirac-type degree condition, while $S\cap Z_i$ remains small and every vertex there has many neighbours in $S\cap A_i$. Chv\'{a}tal's theorem then gives a Hamiltonian cycle in $G[S]$ with probability $1-o(1)$.

Thus each block contributes
$
(1-o(1))2^{|B_i|}
$
cyclic subsets. Since the blocks are disjoint, these contributions are distinct apart from the empty set. Hence
$$
\operatorname{Cyc}(G)
\ge
(1-o(1))\sum_{i=1}^q2^{|B_i|}
\ge
(q-o(1))2^{n/q}.
$$
This is the only case in which the sharp leading coefficient $q$ has to be tracked carefully.

\subsection*{Case (c): $q+1$ components of size about $|R|/(q+1)$}
At first sight this seems smaller than the desired scale $2^{n/q}$. However, the regularity of $G$ prevents these blocks from being completely isolated in the original graph. As in Case (b), let $B_1,\ldots,B_{q+1}$ be the vertex blocks obtained from the components of $R$ together with the exceptional vertices. A global degree-counting argument shows that, for some distinct $i,j\in[q+1]$, the bipartite graph induced by the edges between $B_i$ and $B_j$ contains a large matching. From this point on, we focus only on these two blocks.

Two disjoint cross-block edges from this large matching are enough. If necessary, we extend them to two short paths whose endpoints lie in the dense cores $A_i\subseteq B_i$ and $A_j\subseteq B_j$. Since each core has minimum degree greater than half its order, there are $2^{|A_i|-O(1)}$ and $2^{|A_j|-O(1)}$ choices for Hamiltonian paths between the prescribed endpoint pairs inside $A_i$ and $A_j$, respectively. Combining these two internal Hamiltonian paths with the two fixed connecting paths gives a cycle. Hence
$$
\operatorname{Cyc}(G)\ge 2^{|A_i|+|A_j|-O(1)}
=2^{(2/(q+1)-o(1))n}.
$$
Since $q\ge2$, we have $2/(q+1)>1/q$, and therefore this gives an exponential surplus.

\medskip

\noindent\textbf{Notations}
The expression $0<a\ll b\ll c<1$ means that the constants are chosen from right to left. For instance, one first fixes a constant $c<1$, then chooses $b$ sufficiently small as a function of $c$, and finally chooses $a>0$ sufficiently small as a function of $b$.
We write $x = y \pm z$ to denote the interval $y - z \le x \le y + z$. For two disjoint sets $A, B \subseteq V(G)$, we write $e_G(A,B)$ for the number of edges between them, and define the edge density as $d_G(A,B) = \frac{e_G(A,B)}{|A||B|}$. The subscript is omitted when the graph is clear.

\section{Reduced Graph Trichotomy}\label{sec:dichotomy}

The core objective of this section is to establish a crucial structural trichotomy lemma. 
\begin{lemma}[Reduced graph trichotomy]\label{lem:trichotomy}
Let $q \ge 2$ be an integer and  suppose $1/m\ll \eta\ll\tau\ll 1/q$.  Let $R$ be an $m$-vertex  graph with
 $\delta(R) \ge \left(\frac{1}{q+1} - \eta\right)m.$
Then at least one of the following holds:
\begin{enumerate}
    \item[(a)] $R$ contains a connected matching covering at least $(1/q + \tau)m$ vertices;
    \item[(b)] $R$ has exactly $q$ connected components $Y_1, \dots, Y_q$, each of size $\left(\frac{1}{q} \pm 2q\tau\right)m;$
    \item[(c)] $R$ has exactly $q+1$ connected components $X_1, \dots, X_{q+1}$, each of size $\left(\frac{1}{q+1} \pm \tau\right)m.$
\end{enumerate}
  \end{lemma}

The following lemma of Dirac relates the minimum degree to the length of a longest path.

 \begin{lemma}[Dirac]\label{lem:long_path}
Every connected graph $H$ has a path of order at least $\min\{|H|,2\delta(H)+1\}$.
\end{lemma}

\begin{proof}[Proof of Lemma \ref{lem:trichotomy}]
Let the components of $R$ be $C_1,\ldots,C_s$. Then
$|C_j| \ge\delta(R) + 1\ge \left(\frac{1}{q+1} - 2\eta\right)m$  for any $1\le j\le s$,   which immediately implies $s \le q+1$ since $\eta \ll 1/q$.

Suppose first that $s \le q-1$. Let $C$ be a largest component. Then $|C| \ge m/(q-1)$. By Lemma \ref{lem:long_path}, $C$ contains a path on at least
$$ \min\left\{\frac{m}{q-1}, \left(\frac{2}{q+1} - 3\eta\right)m\right\} \ge \left(\frac{1}{q} + 2\tau\right)m $$
vertices, since $\eta\ll \tau\ll 1/q$. Taking alternating edges on this path gives a connected matching covering  all but at most one vertex of the path, and hence  at least $(1/q + \tau)m$ vertices  for $m$ sufficiently large. Thus (a) holds.

Now suppose that $s = q$. If some component has size at least $(1/q + 2\tau)m$, the same long path argument guarantees a connected matching covering at least $(1/q + \tau)m$ vertices. Hence we may assume that every component  has size at most $(1/q + 2\tau)m$. Since the $q$ component sizes sum exactly to $m$, every component then has size at least
$m - (q-1)\left(\frac{1}{q} + 2\tau\right)m = \left(\frac{1}{q} - 2(q-1)\tau\right)m. $
Thus (b) holds.

Finally, suppose that $s = q+1$. Since $\sum_{j\in[q+1]} |C_j| = m,$
the uniform lower bound on the components forces an upper bound on each $|C_i|$:
$|C_i| \le m - q\left(\frac{1}{q+1} - 2\eta\right)m = \left(\frac{1}{q+1} + 2q\eta\right)m.$
Combining this with the lower bound, and also the fact that $\eta\ll \tau$, yields 
$|C_i| = \left(\frac{1}{q+1} \pm \tau\right)m$ for all $i$, thus establishing (c).
\end{proof}

\section{The Non-Degenerate Case: Connected Matchings}\label{sec:non_degenerate}

We first record a few standard tools on regular pairs and then use them to convert a connected matching in the reduced graph into many cyclic subsets. For $\rho < 1$, a pair of disjoint vertex sets $(A, B)$ is \emph{$\rho$-regular} if every $A' \subseteq A$ and $B' \subseteq B$ with $|A'| \ge \rho|A|$ and $|B'| \ge \rho|B|$ satisfies
$$ |d(A', B') - d(A, B)| \le \rho. $$
Such a pair is called \emph{$(\rho,\mu)$-super-regular} if, in addition, every vertex $a\in A$ satisfies $d(a,B)\ge \mu |B|$ and every vertex $b\in B$ satisfies $d(b,A)\ge \mu |A|$. Observe that this condition already implies $d(A,B)\ge \mu$.

The following proposition is the main target in this section.

\begin{proposition}\label{prop:case_a}
Let $1/L\ll \rho,1/m<1$ and $\rho\ll \mu,\zeta$. Suppose $G$ contains $m$ pairwise disjoint clusters $V_1,\dots,V_m$ of size $L$, and let $R$ be a graph on vertex set $[m]$. Assume that for every edge $ij\in E(R)$, the pair $(V_i,V_j)$ is $\rho$-regular in $G$ with density at least $\mu$. If $R$ admits a connected matching covering $t\ge4$ vertices, then
$
\operatorname{Cyc}(G)\ge 2^{tL-\zeta mL}.
$
\end{proposition}

The proof has three steps: we first clean the regular pairs into super-regular cores, then connect the matching edges by short paths, and finally use a probabilistic version of the Blow-up Lemma to obtain exponentially many internal Hamiltonian paths.

\begin{lemma}\label{lem:slicing}
Let $0<\rho\ll \mu< 1$, and let $(A,B)$ be a $\rho$-regular pair of density at least $\mu$.
\begin{enumerate}
\item[(i)] For every $B'\subseteq B$ with $|B'|\ge \rho |B|$, all but at most $\rho |A|$ vertices $a\in A$ satisfy
$|N(a)\cap B'|\ge (\mu-\rho)|B'|.$

\item[(ii)] If $A'\subseteq A$ and $B'\subseteq B$ satisfy $|A'|\ge \alpha |A|$ and $|B'|\ge \alpha |B|$ for some $\alpha>\rho$, then $(A',B')$ is $(2\rho/\alpha)$-regular and has density at least $\mu-\rho$.

\item[(iii)] If $|A|=|B|=L$, then there exist $A_0\subseteq A$ and $B_0\subseteq B$ such that
$|A_0|=|B_0|\ge (1-\rho)L$
and $(A_0,B_0)$ is $(4\rho,\mu/2)$-super-regular.

\end{enumerate}
\end{lemma}

\begin{proof}
Parts (i) and (ii) are standard consequences of regularity. For (iii), apply (i) to $B'=B$ and symmetrically to $A'=A$. After deleting at most $\rho L$ vertices from each side and then balancing the two sides, we obtain subsets $A_0\subseteq A$ and $B_0\subseteq B$ with $|A_0|=|B_0|\ge(1-\rho)L$. The inherited regularity follows from (ii), and every remaining vertex loses at most $\rho L$ neighbours when passing to the cleaned opposite side. Since $\rho\ll\mu$, the resulting pair is $(4\rho,\mu/2)$-super-regular.
\end{proof}

Next, we establish a tool to connect these super-regular cores. The following lemma guarantees that we can embed short paths through the regular pairs, even while aggressively avoiding a small set of already utilized vertices.

\begin{lemma}\label{lem:path_embed}
Let $0<\rho\ll \mu<1$. Then, for all  integers $b,L\ge 1$ satisfying $1/L\ll \rho,1/b$, the following holds.
Let $Z_0,Z_1,\dots,Z_{\ell}$ be pairwise disjoint clusters of size $L$, where ${\ell}\ge 1$, such that each consecutive pair $(Z_{j-1},Z_j)$ is $\rho$-regular with density at least $\mu$. Let $F$ be a set of forbidden vertices with
$
|F\cap Z_j|\le b$ for all $0\le j\le \ell$. If $A\subseteq Z_0\setminus F$ and $B\subseteq Z_{\ell}\setminus F$ satisfy $|A|,|B|\ge  L/3$, then $G-F$ contains a path
$
z_0z_1\dots z_{\ell}
$
such that $z_0\in A$, $z_{\ell}\in B$, and $z_j\in Z_j$ for every $0\le j\le \ell$.
\end{lemma}

\begin{proof}
We define the sets backwards. Put $S_{\ell}=B$. Having defined $S_{j+1}$, let $S_j$ be the set of vertices in $Z_j\setminus F$ with at least one neighbour in $S_{j+1}$.

We claim that $|S_j|\ge  L/6$ for every $0\le j\le {\ell}$, and moreover $|S_j|\ge L-\rho L-b$ for every $0\le j<{\ell}$.
The first assertion is clear for $j={\ell}$, since $S_{\ell}=B$ and $|B|\ge  L/3$. Now suppose that $j<{\ell}$ and $|S_{j+1}|\ge L/6$. Lemma~\ref{lem:slicing}(i), applied to the pair $(Z_j,Z_{j+1})$ and the set $S_{j+1}$, implies that all but at most $\rho L$ vertices of $Z_j$ have a neighbour in $S_{j+1}$. Removing the forbidden vertices gives
$
|S_j|\ge L-\rho L-b\ge L/6.
$
This proves the claim.

Since $|Z_0\setminus S_0|\le \rho L+b< |A|$, we have $A\cap S_0\neq\varnothing$. Choose $z_0\in A\cap S_0$.
Now suppose that $z_0,\dots,z_j$ have been chosen with $z_j\in S_j$. By the definition of $S_j$, the vertex $z_j$ has a neighbour in $S_{j+1}$; choose one such neighbour as $z_{j+1}$. Since the clusters $Z_0,\dots,Z_{\ell}$ are pairwise disjoint, the vertices chosen in this way are distinct. Repeating this for $0\le j<{\ell}$ gives a path
$
z_0z_{1}\dots z_{\ell}
$
in $G-F$, with $z_0\in A$, $z_{\ell}\in B$, and $z_j\in Z_j$ for every $0\le j\le {\ell}$.
\end{proof}

Inside each super-regular matching edge, we now embed Hamiltonian paths using the following Blow-up Lemma by Koml\'{o}s, S\'{a}rk\"{o}zy and Szemer\'{e}di \cite{KSS97}.

\begin{lemma}[Koml\'{o}s, S\'{a}rk\"{o}zy and Szemer\'{e}di]\label{lem:blowup_ham} 
Let $1/N\ll\rho\ll\mu<1$.
If $(A, B)$ is $(\rho, \mu)$-super-regular and $|A| = |B| = N$, then for any $a \in A$ and $b \in B$, the graph $G[A, B]$ contains a Hamiltonian path from $a$ to $b$.
\end{lemma}

To establish an exponential lower bound, finding a single path is not enough. Instead, we use a probabilistic argument to prove that randomly selecting subsets within a super-regular pair yields exponentially many Hamiltonian paths.
\begin{lemma}\label{lem:abundant_ham}
Let $1/N\ll \rho\ll \mu,\zeta<1$. If $(X,Y)$ is a $(\rho,\mu)$-super-regular pair with $|X|=|Y|=N$, then for any $x\in X$ and $y\in Y$, there are at least
$
2^{2N-\zeta N}
$
pairs of sets $(A,B)$ with $x\in A\subseteq X$, $y\in B\subseteq Y$ such that $G[A,B]$ contains a Hamiltonian path from $x$ to $y$.
\end{lemma}

\begin{proof}
We only count pairs $(A,B)$ with
$x\in A\subseteq X$, $y\in B\subseteq Y$, and $|A|=|B|=\lfloor N/2\rfloor$.
The number of such pairs is
$$
T=\binom{N-1}{\lfloor N/2\rfloor-1}^2\ge 2^{2N-\zeta N/2}.
$$

Choose such a pair $(A,B)$ uniformly at random.

\begin{claim*}
With probability $1-o(1)$, the pair $(A,B)$ is $(6\rho,\mu/4)$-super-regular.
\end{claim*}

Assume the claim. Then, by Lemma~\ref{lem:blowup_ham}, $G[A,B]$ contains a Hamiltonian path from $x$ to $y$. Hence a $1-o(1)$ proportion of the $T$ pairs are valid. Since $T\ge 2^{2N-\zeta N/2}$, the number of valid pairs is at least
$
2^{2N-\zeta N}.
$
\end{proof}
\begin{poc}
Put $r=\lfloor N/2\rfloor$. Since $|A|=|B|=r\ge N/3$, Lemma~\ref{lem:slicing}(ii) implies deterministically that $(A,B)$ is $6\rho$-regular and has density at least $\mu-\rho\ge \mu/2$.

It remains to verify the  minimum-degree condition. Fix $u\in X$. Since $d(u,Y)\ge \mu N$ and $B$ is a uniformly chosen $r$-subset of $Y$ containing $y$, the random variable $|N(u)\cap B|$ is hypergeometric up to the fixed inclusion of $y$, and has expectation at least $3\mu r/4$ for $N$ sufficiently large. Hence, by the Chernoff bound, there is a constant $c=c(\mu)>0$ such that
$
\Pr\bigl(|N(u)\cap B|<\mu r/4\bigr)\le e^{-cN}.
$
A union bound over all $u\in X$ gives, with probability $1-o(1)$,
$
|N(u)\cap B|\ge \mu r/4$ for every $u\in X$.
The symmetric argument gives $|N(v)\cap A|\ge\mu r/4$ for every $v\in Y$ with probability $1-o(1)$. Therefore every vertex of $A$ has at least $\mu |B|/4$ neighbours in $B$, and   every vertex of $B$ has at least $\mu |A|/4$ neighbours in $A$. Thus $(A,B)$ is $(6\rho,\mu/4)$-super-regular with probability $1-o(1)$.
\end{poc}

Gathering Lemmas~\ref{lem:slicing}, \ref{lem:path_embed} and \ref{lem:abundant_ham} together, we are now ready to prove Proposition \ref{prop:case_a}. 

\begin{proof}[Proof of Proposition~\ref{prop:case_a}]
Let
$
{P_1Q_1,\dots,P_sQ_s}
$
be the connected matching in $R$, where we identify each vertex of $R$ with the corresponding cluster of $G$. Then $s=t/2\ge 2$. 
 For each matching edge $P_iQ_i$, Lemma~\ref{lem:slicing}(iii) gives subsets
$
P_i^0\subseteq P_i$ and $Q_i^0\subseteq Q_i
$
with
$$
|P_i^0|=|Q_i^0|\ge L_1:= (1-\rho)L
$$
such that $(P_i^0,Q_i^0)$ is $(4\rho,\mu/2)$-super-regular.

\begin{claim}
There exist vertex-disjoint paths $C_1,\dots,C_s$ in $G$ such that, for each $i\in[s]$, the path $C_i$ joins a vertex $v_i\in Q_i^0$ to a vertex $u_{i+1}\in P_{i+1}^0$, where indices are taken cyclically. Moreover,
$
\left|\left(\cup_{i\in[s]}V(C_i)\right)\cap V_j\right|\le m^2
$
for every cluster $V_j$.
\end{claim}

\begin{poc}
For each $i\in[s]$, choose a shortest path in $R$ from $Q_i$ to $P_{i+1}$. Such a path exists because all matching edges lie in one connected component of $R$, and it has length at most $m$.

We embed these paths one by one. Suppose that some connecting paths have already been embedded in $G$, and let $U$ be the set of their vertices. Since at most $s\le m/2$ paths are embedded and each uses at most $m$ vertices in $G$, we have
  $|U\cap V_j|\le m^2 \le \rho L$ for every cluster $V_j$. 
 Hence the endpoint sets
$
Q_i^0\setminus U
$ and $
P_{i+1}^0\setminus U
$
have size at least $L/3$.

Applying Lemma~\ref{lem:path_embed} with $b=m^2$, we obtain a path in $G-U$ from some vertex $v_i\in Q_i^0$ to some vertex $u_{i+1}\in P_{i+1}^0$, following the chosen cluster path in $R$. Repeating this for all $i\in[s]$ gives the desired paths.
\end{poc}
Let $F$ be the set of internal vertices of the paths $C_i$. Then the claim gives
$$
|F\cap V_j|\le m^2\le \rho L
$$
for every cluster $V_j$.

We now refine the super-regular pairs. For each $i\in[s]$, delete from $P_i^0\cup Q_i^0$ all vertices of $F$, while keeping the endpoints $u_i\in P_i^0$ and $v_i\in Q_i^0$. Then discard additional vertices from the larger side if necessary, so as to obtain subsets
$
P_i^*\subseteq P_i^0$ and $Q_i^*\subseteq Q_i^0
$
with
$$
u_i\in P_i^*,\qquad v_i\in Q_i^*,\qquad |P_i^*|=|Q_i^*|=:L_i.
$$
Since at most $\rho L$ vertices are lost from each cluster, we have
$
L_i\ge (1-2\rho)L.
$

We claim that $(P_i^*,Q_i^*)$ is $(16\rho,\mu/4)$-super-regular. By Lemma~\ref{lem:slicing}(ii), it is $16\rho$-regular. Moreover, every vertex of $P_i^*$ loses at most $2\rho L$ neighbours when passing from $Q_i^0$ to $Q_i^*$, and hence has degree at least
$
\frac{\mu}{2}|Q_i^0|-2\rho L\ge \frac{\mu}{4}|Q_i^*|
$
into $Q_i^*$. The same argument applies symmetrically to vertices of $Q_i^*$.

By Lemma~\ref{lem:abundant_ham}, applied with parameters $16\rho$, $\mu/4$, and $\zeta /8$, for each $i\in[s]$ there are at least
$
2^{2L_i-\zeta L_i/8}
$
pairs of sets $(A_i,B_i)$ with
$
u_i\in A_i\subseteq P_i^* $ and $ v_i\in B_i\subseteq Q_i^*
$
such that $G[A_i,B_i]$ contains a Hamiltonian path from $u_i$ to $v_i$.

For every choice of such pairs over all $i\in[s]$, concatenate the Hamiltonian path from $u_i$ to $v_i$ inside $G[A_i,B_i]$ with the fixed connecting path $C_i$ from $v_i$ to $u_{i+1}$, for every $i\in[s]$. This gives a cycle. Since the sets $P_i^*,Q_i^*$ and $F$ are disjoint, different choices of the pairs $(A_i,B_i)$ give different vertex sets. Therefore
  $$\operatorname{Cyc}(G)
\ge \prod_{i=1}^s 2^{2L_i-\zeta L_i/8}\ge 2^{s(2-\zeta/8)(1-2\rho)L} \ge 2^{tL-\zeta mL},
$$
 where the last inequality uses $t=2s$, $s\le m/2$, and $\rho\ll\zeta$.
 This proves the proposition.
\end{proof}

\section{The Balanced Dense Case: $q$ Components}\label{sec:balanced_dense}

We now consider Case (b) of the trichotomy, where the reduced graph decomposes into exactly $q$ components of nearly equal size. 

\begin{proposition}\label{prop:case_b}
Let $1/n\ll \sigma\ll \alpha,\xi,1/q\le 1$, where $q\ge 2$ is an integer. Suppose $G$ is an $n$-vertex graph whose vertex set is partitioned into $q$ blocks $B_1,\dots,B_q$ such that
$
|B_i|\ge {n}/{(2q)}
$
for all $i\in[q]$. Assume that each block $B_i$ admits a partition $B_i=A_i\cup Z_i$ satisfying
$$
|Z_i|\le \sigma |B_i|,\quad
\delta(G[A_i])\ge \left(\frac{1}{2}+5\alpha\right)|A_i|,\quad
d_G(z,A_i)\ge 5\alpha |A_i|\quad\text{for every }z\in Z_i.
$$
Then
$
\operatorname{Cyc}(G)\ge (q-\xi)2^{n/q}.
$
\end{proposition}

We first prove a counting lemma for a single block.

\begin{lemma}\label{lem:dense_block}
Let $1/N\ll \sigma\ll \alpha,\xi<1$. Let $H$ be an $N$-vertex graph with a partition $V(H)=A\cup Z$ such that $|Z|\le \sigma N$. Assume that
$$
\delta(H[A])\ge \left(\frac{1}{2}+5\alpha\right)|A|
\quad\text{and}\quad
d_H(z,A)\ge 5\alpha |A| \text{ for every } z\in Z.
$$
Then
$
\operatorname{Cyc}(H)\ge (1-\xi)2^{N}.
$
\end{lemma}

We shall use the following theorem of Chv\'{a}tal~\cite{Chvatal}. In the application below, we only need the immediate consequence that a graph with degree sequence $d_1\le\dots\le d_m$ is Hamiltonian whenever $d_k>k$ for every $k<m/2$.

\begin{theorem}[Chv\'{a}tal]\label{thm:chvatal}
Let $F$ be a graph on $m \ge 3$ vertices with degree sequence $d_1 \le d_2 \le \dots \le d_m$. For every integer $k < m/2$, if
$d_k \le k$ implies $d_{m-k} \ge m-k$, then $F$ is Hamiltonian.
\end{theorem}

\begin{proof}[Proof of Lemma \ref{lem:dense_block}]

Generate a random subset $S\subseteq V(H)$ by including each vertex
independently with probability $1/2$. Put
\(
 S_A=S\cap A\) {and} \( S_Z=S\cap Z.
\)

\begin{claim}
With probability $1-o(1)$, the following properties hold simultaneously:
\begin{enumerate}
\item[(i)] $|S_A|=(1/2\pm \alpha)|A|$;
\item[(ii)] for every $a\in S_A$,
$d_H(a,S_A)\ge (1/2+2\alpha)|S_A|$;
\item[(iii)] for every $z\in S_Z$,
$d_H(z,S_A)\ge 2\alpha |S_A|$;
\item[(iv)] $|S_Z|<\alpha |S_A|$.
\end{enumerate}
\end{claim}

\begin{poc}
By Chernoff's inequality, with probability $1-o(1)$, (i) holds. For (ii),
fix $a\in A$. Since
$
\mathbb E[d_H(a,S_A)]\ge
\left(\frac{1}{2}+5\alpha\right)\frac{|A|}{2}=
\left(\frac{1}{4}+\frac{5\alpha}{2}\right)|A|,
$
it follows from Chernoff's inequality that
$
\Pr\left(d_H(a,S_A)<\left(\frac{1}{4}+2\alpha\right)|A|\right)=o(1/N).
$
A union bound over all $a\in A$ shows that, with probability $1-o(1)$,
$
d_H(a,S_A)\ge \left(\frac{1}{4}+2\alpha\right)|A| \ge \left(\frac{1}{2}+2\alpha\right)|S_A|$
 for every $a\in A$.

The same Chernoff and union-bound argument, using $d_H(z,A)\ge 5\alpha |A|$ for every $z\in Z$, gives   with probability $1-o(1)$  that $d_H(z,S_A)\ge (\alpha+3\alpha^2)|A|$ for every $z\in Z$.  On the event in (i), this is at least $2\alpha |S_A|$, and hence (iii) holds.

Finally, on the event in (i), the lower bound $|S_A|\ge (1/2-\alpha)|A|$, together with $|Z|\le \sigma N$, gives
$
|S_Z|\le |Z|\le \sigma N<\alpha |S_A|.
$
Thus all four properties hold simultaneously with probability $1-o(1)$.
\end{poc}

Condition on these four properties. Let $r=|S|$ and let
 $d_1\le d_2\le \dots\le d_r$
be the degree sequence of $H[S]$. We verify that
 $d_i>i$ for every $i<r/2$.

First suppose $i\le |S_Z|$. By (ii) and (iii), every vertex of $S$ has
degree in $H[S]$ at least $2\alpha |S_A|$. Hence
 $d_i\ge 2\alpha |S_A|>|S_Z|\ge i.$

Next suppose $|S_Z|<i<r/2$. Among the $i$ vertices of smallest degree,
at least one lies in $S_A$. Therefore, by (ii),
\[
 d_i\ge \left(\frac{1}{2}+2\alpha\right)|S_A|>\frac{|S_A|+|S_Z|}{2}
=\frac{r}{2}>i.\]
Consequently $d_i>i$ for all $i<r/2$.  Chv\'{a}tal's
Theorem~\ref{thm:chvatal} implies that $H[S]$ is Hamiltonian with probability at least $1-o(1)$, which gives
$ \operatorname{Cyc}(H)\ge (1-o(1))2^N\ge  (1-\xi)2^N,$
as required.
\end{proof}
\medskip

\begin{proof}[Proof of Proposition~\ref{prop:case_b}]
Applying Lemma~\ref{lem:dense_block} to each $G[B_i]$ gives
$
\operatorname{Cyc}(G[B_i])\ge \left(1-\frac{\xi}{2q}\right)2^{|B_i|}$
for every $i\in[q]$.
Every cyclic subset of $G[B_i]$ is also a cyclic subset of $G$. Moreover, apart from the empty set, the cyclic subsets counted inside different blocks are distinct. Hence,
$$
\operatorname{Cyc}(G)
\ge \sum_{i=1}^q\left(\left(1-\frac{\xi}{2q}\right)2^{|B_i|}-1\right)
= \left(1-\frac{\xi}{2q}\right)\sum_{i=1}^q 2^{|B_i|}-q 
\ge \left(1-\frac{\xi}{2q}\right)q2^{n/q}-q 
\ge (q-\xi)2^{n/q}.
$$ 
This proves the proposition.
\end{proof}

\section{The Near-Critical Case: $q+1$ Components}\label{sec:near_critical}

We now consider Case (c) of the trichotomy, where the reduced graph separates into $q+1$ components. 
We prove the following proposition, allowing $k=2$ for later use in the proof of Proposition~\ref{prop:dirac_rate}.
\begin{proposition}\label{prop:case_c}
Let $k\ge 2$ be an integer and let $C>0$. Suppose
$
1/n\ll \theta,\rho,\lambda\ll 1/k,1/C.
$
Let $G$ be an $n$-vertex $d$-regular graph satisfying
$
k(d+1)>n,
$
and suppose that its vertex set is partitioned as
$
V(G)=V_0\cup A_1\cup\dots\cup A_k
$
such that the following hold:
\begin{enumerate}
\item[(i)] $|V_0|\le \rho n$;
\item[(ii)] $ |A_i|= (1/k\pm\lambda)n$ for every $i\in[k]$;
\item[(iii)] $d_G(v,A_i)\ge |A_i|-\theta n$ for every $i\in[k]$ and every $v\in A_i$.
\end{enumerate}
Then
$$
\operatorname{Cyc}(G)\ge 
\begin{cases}
2^{(1-3\lambda)n}, & \text{if } k=2,\\
C\cdot 2^{n/(k-1)}, & \text{if } k\ge 3.
\end{cases}
$$
\end{proposition}

The proof has three ingredients. First, a global edge-counting argument forces a large matching between two blocks. Second, inside each dense block, there are exponentially many Hamiltonian paths between any prescribed pair of vertices. Finally, two cross-block edges, together with these internal paths, form cycles.

\begin{lemma}\label{lem:compensating}
Let $k\ge 2$ be an integer and let $1/n\ll c\ll 1/k$. Let $G$ be an $n$-vertex $d$-regular graph, and let
$
V(G)=A_1\cup\dots\cup A_k
$
be a partition such that $|A_i|\ge n/(3k)$ for every $i\in[k]$. If
$
k(d+1)-n\ge 1,
$
then there exist distinct $i,j\in[k]$ such that the bipartite graph $G[A_i,A_j]$ contains a matching of size at least $c\sqrt n$.
\end{lemma}

\begin{proof}
Let $a_i=|A_i|$ and define
$
\Delta_i=d-a_i+1.
$
Then
$\sum_{i=1}^k\Delta_i=k(d+1)-n\ge1.
$
We assume that
$
c\le k^{-9}.
$ Suppose, for a contradiction, that
$
\nu(G[A_i,A_j])<c\sqrt n$
for every distinct $i,j\in[k]$.
By K\H{o}nig's theorem, for every pair $i<j$, there is a vertex cover
$
D_{ij}\subseteq A_i\cup A_j
$
of size less than $c\sqrt n$ that covers all edges between $A_i$ and $A_j$.
For each $i\in[k]$, define
$
C_i:=A_i\cap\bigcup_{j\ne i}D_{ij}, $ $
B_i:=A_i\setminus C_i,
$
and put
$
C:=\sum_{i=1}^k |C_i|.
$
Then
$
C<\binom{k}{2}c\sqrt n.
$

\begin{claim}
For every $i\in[k]$, the set $B_i$ is non-empty, there are no edges between $B_i$ and $B_j$ whenever $i\ne j$, and 
$
|\Delta_i|\le kC.
$
\end{claim}

\begin{poc}
Since $C<\binom{k}{2}c\sqrt n$ and $|A_i|\ge n/(3k)$, every $B_i$ is non-empty.
By the definition of $C_i$, every edge between $A_i$ and $A_j$ is covered by $C_i\cup C_j$. Hence there are no edges between $B_i$ and $B_j$ for distinct $i,j$.

Now fix $i\in[k]$ and $v\in B_i$. Since there are no edges between $B_i$ and $B_j$ for $j\ne i$, all neighbours of $v$ outside $A_i$ must lie in $\bigcup_{j\ne i}C_j$. Therefore
$
d\le (a_i-1)+(C-|C_i|).
$
Thus
$
\Delta_i=d-a_i+1\le C.
$
On the other hand, since $\sum_{j=1}^k\Delta_j\ge1$, we have
$
\Delta_i\ge 1-\sum_{j\ne i}\Delta_j> -(k-1)C.
$
Therefore $|\Delta_i|\le kC$ for every $i\in[k]$.
\end{poc}

\begin{claim}
We have
$
\sum_{i=1}^k |B_i|\Delta_i
\le
\sum_{i=1}^k |C_i|(d-|B_i|).
$
\end{claim}

\begin{poc}
For each $i$, define
$
h_i:=\sum_{v\in B_i}\bigl(a_i-1-d_{G[A_i]}(v)\bigr).
$
This is the number of ordered missing incidences from $B_i$ to $A_i$. We have
\begin{align*}
e(B_i,V(G)\setminus A_i)
=
\sum_{v\in B_i}\bigl(d-d_{G[A_i]}(v)\bigr) 
=
\sum_{v\in B_i}\bigl(\Delta_i+a_i-1-d_{G[A_i]}(v)\bigr) =
|B_i|\Delta_i+h_i.
\end{align*}
Similarly,
\begin{align*}
e(C_i,V(G)\setminus A_i)
\le |C_i| d-e(C_i,B_i) =|C_i|(d-|B_i|)+|B_i||C_i|-e(C_i,B_i) \le |C_i|(d-|B_i|)+h_i.
\end{align*}

Every edge leaving $B_i$ for another block must land in some $C_j$. Hence, after summing over all $i$, the edges counted by
$
\sum_{i=1}^k e(B_i,V(G)\setminus A_i)
$
are bounded above by the edges counted by
$
\sum_{i=1}^k e(C_i,V(G)\setminus A_i).
$
Using the two estimates above,  we obtain the claimed inequality.
\end{poc}

By the second claim, substituting 
$
|B_i|=a_i-|C_i|=d+1-\Delta_i-|C_i|,
$
we have
$$
\sum_{i=1}^k (d+1-\Delta_i-|C_i|)\Delta_i
\le
\sum_{i=1}^k |C_i|(\Delta_i+|C_i|-1).
$$
Rearranging, we obtain
$$
(d+1)\sum_{i=1}^k\Delta_i
\le
\sum_{i=1}^k\bigl(\Delta_i^2+2|C_i|\Delta_i+|C_i|^2-|C_i|\bigr)
\le
\sum_{i=1}^k\bigl(|\Delta_i|+|C_i|\bigr)^2.
$$
The left-hand side is larger than $n/k$, since
$
\sum_{i=1}^k\Delta_i=k(d+1)-n\ge1
$
and $d+1>n/k$.
On the other hand, by the first claim, $|\Delta_i|\le kC$. Hence
 the right-hand side is at most $
k(k+1)^2C^2<n/k, $
a contradiction. Thus there must exist distinct $i,j\in[k]$ such that
$
\nu(G[A_i,A_j])\ge c\sqrt n.
$
This completes the proof.
\end{proof}

\begin{lemma}\label{lem:near_complete_ham}
Let $k\ge 2$ be an integer and let $1/n\ll \theta\ll 1/k$. Let $H$ be an $N$-vertex graph with
$
N\ge n/(3k)
$
and
$
\delta(H)\ge N-\theta n.
$
Then, for any distinct $x,y\in V(H)$, there are at least
$
 2^{N-3}
$
sets $S\subseteq V(H)$ with $x,y\in S$ such that $H[S]$ contains a Hamiltonian path from $x$ to $y$.
\end{lemma}

\begin{proof}
We only count subsets $S\subseteq V(H)$ containing $x$ and $y$ with $|S|\ge N/3$. The total number of such $S$ is at least $$\sum_{N/3 \le i\le N-2}\binom{N-2}{i}\ge 2^{N-3}.$$  For every such $S$, we have
$$
\delta(H[S])\ge |S|-1-(N-1-\delta(H))\ge |S|-\theta n>\frac{|S|}{2}+1.
$$
{Let $T=S\setminus\{x,y\}$.  Then every vertex of $T$ has degree in $H[T]$ greater than $|S|/2-1=|T|/2$.  By Dirac's Hamilton-cycle theorem~}\cite{Dirac52}, $H[T]$  contains a Hamiltonian cycle $C$.   Moreover,  $x$ and $y$ each have at least $|S|/2$ neighbours   on  $C$.  Orient $C$.  If no successor of a vertex in $N_C(x)$ lay in $N_C(y)$, then $|N_C(y)|\le |C|-|N_C(x)|<|S|/2$, a contradiction.  Hence there are adjacent vertices $x',y'$  on $C$ such that $xx'$ and $yy'$ are edges.  Deleting the edge $x'y'$ from  $C$  and adding $xx'$ and $yy'$ gives a Hamiltonian  path from $x$ to $y$  in $H[S]$.
\end{proof}

\begin{proof}[Proof of Proposition \ref{prop:case_c}]
Assign each $z\in V_0$ to an index $i$ maximizing $d_G(z,A_i)$, let $Z_i$ be the set of vertices assigned to $A_i$, and define  $B_i = A_i \cup Z_i$. The choice of $i$ ensures that
$$ d_G(z, A_i) \ge \frac{1}{k}(d-|V_0|) \ge \frac{1}{k}\left(\frac{n}{k} - 1 - \rho n\right) \ge \frac{n}{2k^2}.$$ 

 Applying Lemma \ref{lem:compensating} to the partition 
$V(G) = B_1 \cup \dots \cup B_k,$
there exist distinct indices $i$ and $j$ such that the bipartite graph $G[B_i, B_j]$ contains a matching of size at least $c\sqrt{n}$, for some $c=c(k)>0$. 

We select two disjoint edges from this matching, say $e_1 = u_1v_1$ and $e_2 = u_2v_2$, with $u_1,u_2 \in B_i$ and $v_1,v_2 \in B_j$. We will extend these two edges into short connecting paths between $A_i$ and $A_j$. For $l \in \{1, 2\}$, if $u_l \in A_i$, we set $x_l = u_l$. Otherwise, $u_l \in Z_i$, and since $u_l$ has at least $n/(2k^2)$ neighbours in $A_i$, we choose  $x_l \in N_G(u_l) \cap A_i$ while avoiding the $O(1)$ vertices already chosen. Similarly, if $v_l \in A_j$, we set $y_l = v_l$; otherwise, we choose $y_l \in N_G(v_l) \cap A_j$ while avoiding the vertices already chosen. Thus $x_1,x_2,y_1,y_2$ are distinct, and  we obtain two disjoint paths $P_1$ and $P_2$ of length at most 3 that connect $x_1, x_2 \in A_i$ to $y_1, y_2 \in A_j$. Let $F$ be the set of internal vertices used by these paths. Note that $|F| \le 4$.

We apply Lemma \ref{lem:near_complete_ham} to  $H_i = G[A_i \setminus F]$. Since at most four vertices are removed, $H_i$ satisfies the same hypothesis after replacing $\theta$ by $2\theta$ and increasing $n_0$ if necessary. This yields at least 
$2^{|A_i| - |F|- 3} \ge 2^{|A_i| -7}$
sets $S_i$ satisfying $\{x_1, x_2\} \subseteq S_i \subseteq A_i \setminus F$ such that $G[S_i]$ has a Hamiltonian path from $x_1$ to $x_2$. 

Likewise, applying the lemma to $H_j = G[A_j \setminus F]$ yields at least $2^{|A_j| - 7}$ sets $S_j$ satisfying $\{y_1, y_2\} \subseteq S_j \subseteq A_j \setminus F$ such that $G[S_j]$ has a Hamiltonian path from $y_1$ to $y_2$. 

For any pair of such subsets $(S_i, S_j)$, we can join their respective Hamiltonian paths with the fixed connectors $P_1$ and $P_2$. This forms a single cycle on the vertex set $S_i \cup S_j \cup F$. Since different pairs $(S_i, S_j)$ yield distinct vertex sets, we have
\begin{align*}
\operatorname{Cyc}(G) &\ge 2^{|A_i|-7}\cdot 2^{|A_j|-7}\ge 2^{(2/k-2\lambda)n-14} \ge \begin{cases}
2^{(1-3\lambda)n}, & \text{if } k=2,\\
C\cdot 2^{n/(k-1)}, & \text{if } k\ge 3.
\end{cases},
\end{align*}
where for $k\ge3$ we use $2/k-1/(k-1)=(k-2)/(k(k-1))>0$, the hierarchy $\lambda\ll 1/k,1/C$, and then take $n$ sufficiently large.
This proves the proposition.
\end{proof}

\section{Proof of the Staircase Theorem}\label{sec:proof}

In the previous sections, we analysed three different shapes our graph could take: non-degenerate (Proposition~\ref{prop:case_a}), balanced dense (Proposition~\ref{prop:case_b}), and near-critical (Proposition~\ref{prop:case_c}). Now, we will put everything together to prove our main result, the Staircase Theorem~\ref{thm:main}.
\begin{proposition}\label{prop:fixed_q}
Let $1/n\ll \xi\ll 1/q<1$, where $q\ge 2$ is an integer.  Let $G$ be an $n$-vertex $d$-regular graph satisfying
$
q(d+1) \le n < (q+1)(d+1).
$
Then
$$
\operatorname{Cyc}(G) \ge (q-\xi)2^{n/q}.
$$
\end{proposition}

We need the following degree form of Szemer\'{e}di's regularity lemma to ensure $R$ inherits the minimum degree of $G$.

\begin{lemma}[Degree form regularity lemma]\label{lem:regularity}
For every $0 < \rho, \mu < 1$ and every integer $m_0$, there exist integers $M$ and $n_0$ such that every graph $G$ on $n \ge n_0$ vertices admits a partition 
$V(G) = V_0 \cup V_1 \cup \dots \cup V_m$
and a spanning subgraph $G' \subseteq G$ satisfying:
\begin{enumerate}
    \item[(i)] $m_0 \le m \le M$;
    \item[(ii)] $|V_0| \le \rho n$ and $|V_1| = \dots = |V_m| =: L$;
    \item[(iii)] every $V_i$ is independent in $G'$;
    \item[(iv)] for every $i \ne j$ the graph $G'[V_i, V_j]$ is either empty or a $\rho$-regular pair of density at least $\mu$;
    \item[(v)] every $v \in V(G)$ satisfies 
    $d_{G'}(v) \ge d_G(v) - (\mu + \rho)n.$
\end{enumerate}
Let $R$ be the reduced graph on vertex set $[m]$, where $ij \in E(R)$ if and only if $G'[V_i, V_j]$ is non-empty. If $G$ is $d$-regular and $\rho \le \mu/10$, then, for $n$ sufficiently large, 
$ \delta(R) \ge \left(\frac{d}{n} - 3\mu\right)m. $
\end{lemma}

\begin{proof}
Properties (i)--(v) are the standard degree form of Szemer\'{e}di's Regularity Lemma~\cite{KS96}.
 We only need to prove the final minimum-degree estimate for $R$.
Fix an index $i \in [m]$ and pick any vertex $v \in V_i$. Since $V_i$ is independent in $G'$, all neighbours of $v$ in $G'$ lie either in $V_0$ or in clusters corresponding to neighbours of $i$ in $R$. Thus, by (v),
$
|V_0|+d_R(i)L\ge d_{G'}(v)\ge d-(\mu+\rho)n,
$
and hence
$
d_R(i)L\ge d-(\mu+\rho)n-|V_0|\ge d-(\mu+2\rho)n.
$

We know the cluster size is $L = (n - |V_0|)/m = (1 \pm \rho)n/m$, and we assume $\rho \le \mu/10$. Substituting this into the inequality yields
$ d_R(i) \ge \left(\frac{d}{n} - 3\mu\right)m$
for all indices $i$.
\end{proof}

We next prove Proposition \ref{prop:fixed_q} by sorting the reduced graph into one of our three cases.
\begin{proof}[Proof of Proposition~\ref{prop:fixed_q}]
Choose constants satisfying
$
0<\tau\ll\theta\ll 1/q.
$
Then choose $0<\alpha\ll\tau$, and choose $0<\sigma\ll\alpha,\xi,1/q$ small enough for Proposition~\ref{prop:case_b}.  Next choose $0<\eta\ll\tau$ and $m_0$   {large enough for} Lemma~\ref{lem:trichotomy}. Finally, choose
$
0<\rho\ll\mu\ll\zeta\ll\eta,\alpha,\sigma.
$

Apply Lemma~\ref{lem:regularity} to $G$ with parameters $\rho,\mu$ and minimum number of clusters $m_0$. This gives a partition $V(G)=V_0\cup V_1\cup\dots\cup V_m$ and a reduced graph $R$. Since $n<(q+1)(d+1)$, we have $d/n>1/(q+1)-o(1)$. Hence, by Lemma~\ref{lem:regularity},
$
\delta(R)\ge \left(\frac{1}{q+1}-\eta\right)m.
$
Apply now Lemma~\ref{lem:trichotomy} to $R$.

\smallskip
\noindent\textbf{Case (a): connected matching with surplus.}
Suppose $R$ contains a connected matching covering $t\ge (1/q+\tau)m$ vertices. By Proposition~\ref{prop:case_a},
$
\operatorname{Cyc}(G)\ge\operatorname{Cyc}(G')\ge 2^{tL-\zeta mL}.
$
Since $mL=n-|V_0|\ge (1-\rho)n$, we have
$$
tL-\zeta mL
\ge \left(\frac{1}q+\tau\right)(n-|V_0|)-\zeta n \ge \frac{n}q+\left(\tau-\zeta-\rho\left(\frac{1}q+\tau\right)\right)n \ge \frac{n}q+\frac{\tau n}{2},
$$
where the last inequality follows from $\zeta,\rho\ll\tau$. Therefore, by the choice of $n_0$,
$$
\operatorname{Cyc}(G)\ge 2^{n/q+\tau n/2}\ge (q-\xi)2^{n/q}.
$$

\smallskip
\noindent\textbf{Case (b): $q$ balanced dense components.}
Suppose $R$ has exactly $q$ components $Y_1,\dots,Y_q$ with $|Y_i|=(1/q\pm 2q\tau)m$ for every $i\in[q]$. For each $i\in[q]$, let $A_i=\bigcup_{j\in Y_i}V_j$. Since $|V_0|\le\rho n$ and $\rho\ll\tau$, we have
$
|A_i|=\left(\frac{1}q\pm 3q\tau\right)n
$
for every $i\in[q]$.
Because there are no edges in $G'$ between distinct components of $R$, every vertex $v\in A_i$ satisfies
$
d_G(v,A_i)\ge d_{G'}(v)-|V_0|\ge d-(\mu+2\rho)n.
$
Using $n<(q+1)(d+1)$, we have $d>n/(q+1)-1$. Therefore, 
$
d_G(v,A_i)\ge \left(\frac{1}{q+1}-\mu-2\rho-o(1)\right)n.
$
By the choice of $\alpha,\tau,\mu,\rho$ and the upper bound $|A_i|\le (1/q+3q\tau)n$, this gives
$
d_G(v,A_i)\ge \left(\frac{1}{2}+5\alpha\right)|A_i|.
$
Hence $\delta(G[A_i])\ge (1/2+5\alpha)|A_i|$ for every $i\in[q]$.

Assign each vertex $z\in V_0$ to an index $i\in[q]$ maximizing $d_G(z,A_i)$. Let $Z_i$ be the set of vertices assigned to $i$, and put $B_i=A_i\cup Z_i$. By averaging over the $q$ blocks,
$$
d_G(z,A_i)\ge \frac{1}q d_G(z,A_1\cup\dots\cup A_q)\ge \frac{1}q(d-|V_0|)\ge \left(\frac{1}{q(q+1)}-\frac{\rho}{q}-o(1)\right)n.
$$
By the choice of $\alpha,\tau,\rho$ and $|A_i|\le (1/q+3q\tau)n$, this is at least $5\alpha |A_i|$ for sufficiently large $n$.

Moreover, $|Z_i|\le |V_0|\le\rho n\le\sigma |B_i|$, and $|B_i|\ge |A_i|\ge n/(2q)$. Thus the partition $V(G)=B_1\cup\dots\cup B_q$ satisfies the hypotheses of Proposition~\ref{prop:case_b}. Therefore
$
\operatorname{Cyc}(G)\ge (q-\xi)2^{n/q}.
$

\smallskip
\noindent\textbf{Case (c): $q+1$ near-critical components.}
Suppose $R$ has exactly $q+1$ components $X_1,\dots,X_{q+1}$ with $|X_i|=(1/(q+1)\pm\tau)m$ for every $i\in[q+1]$. For each $i\in[q+1]$, let $A_i=\bigcup_{j\in X_i}V_j$. Since $\rho\ll\tau$, we have
$
|A_i|=\left(\frac{1}{q+1}\pm 2\tau\right)n
$
for every $i\in[q+1]$.

Again, since there are no edges in $G'$ between distinct components of $R$, every vertex $v\in A_i$ satisfies
$
d_G(v,A_i)\ge d_{G'}(v)-|V_0|\ge d-(\mu+2\rho)n.
$
Using $d>n/(q+1)-1$, $|A_i|\le (1/(q+1)+2\tau)n$, the inequality $3\tau<\theta$, and $\mu,\rho\ll\tau$, we get, for sufficiently large $n$,
$
d_G(v,A_i)\ge |A_i|-\theta n.
$

We now apply Proposition~\ref{prop:case_c} with $k=q+1$, $\lambda=2\tau$, and $C=q-\xi$. Its hypotheses are satisfied by the preceding estimates and by the choice $1/n\ll\theta,\rho,2\tau\ll 1/(q+1),1/(q-\xi)$. Since $q+1\ge3$, Proposition~\ref{prop:case_c} gives
$
\operatorname{Cyc}(G)\ge (q-\xi)2^{n/q}.
$

The three cases exhaust all possibilities, completing the proof.
\end{proof}

We now deduce Theorem~\ref{thm:main} from Proposition~\ref{prop:fixed_q}.

\begin{proof}[Proof of Theorem~\ref{thm:main}]
Fix $\varepsilon>0$ and $\xi>0$. For each integer $q$ with $2\le q\le \lceil1/\varepsilon\rceil$, put $\xi_q=\min\{\xi,1/(100q)\}$, and let $n_q$ be the threshold supplied by Proposition~\ref{prop:fixed_q}  with parameter $\xi_q$. Choose $n_0$ to be the maximum of these finitely many thresholds.

Let $n\ge n_0$, and let $G$ be an $n$-vertex $d$-regular graph with $\varepsilon n\le d< n/2$. Put $q=\lfloor n/(d+1)\rfloor\ge 2$. Since $d\ge\varepsilon n$, we have $q\le \lceil1/\varepsilon\rceil$. By the definition of $q$,
$
q(d+1)\le n<(q+1)(d+1).
$
Therefore Proposition~\ref{prop:fixed_q} applies with parameter $\xi_q$ and yields
$
\operatorname{Cyc}(G)\ge (q-\xi_q)2^{n/q}\ge (q-\xi)2^{n/q}.
$ 
This proves the theorem.
\end{proof}

\section{The boundary case}\label{sec:q1_sharpness}
We finally treat the boundary case $d=n/2$. 

\begin{lemma}\label{lem:boundary_dichotomy}
Let $1/m\ll \eta\ll \tau\ll 1$, and let $R$ be an $m$-vertex graph with
$
\delta(R)\ge \left(\frac{1}{2}-\eta\right)m.
$
Then at least one of the following holds:
\begin{enumerate}
\item[(a)] $R$ contains a connected matching covering at least $(1-\tau)m$ vertices;
\item[(b)] $R$ has exactly two connected components $Y_1,Y_2$, and
$|Y_i|=(1/2\pm\tau)m$ for $i=1,2$.
\end{enumerate}
\end{lemma}

\begin{proof}
Let $C_1,\dots,C_s$ be the connected components of $R$. Since every component has size at least $\delta(R)+1$, we have
$
|C_j|\ge \left(\frac{1}{2}-2\eta\right)m$
for every $j\in[s]$,
and hence $s\le 2$.

If $s=1$, then $R$ is connected. By Lemma~\ref{lem:long_path}, $R$ contains a path on at least
$ (1-3\eta)m$
vertices. Taking alternating edges on this path covers all but at most one vertex of the path, and therefore gives a connected matching covering at least $(1-\tau)m$ vertices  for $m$ sufficiently large. Thus (a) holds.

If $s=2$, then $|C_i|\ge (1/2-2\eta)m$ for $i=1,2$. Since $|C_1|+|C_2|=m$, we also have $|C_i|\le (1/2+2\eta)m=(1/2\pm\tau)m$ for $i=1,2$, so (b) holds.
\end{proof}

\begin{proof}[Proof of Proposition~\ref{prop:dirac_rate}]
Choose constants satisfying
$
0<\tau,\zeta\ll\xi$ and $ 0<\rho\ll\mu\ll\eta,\tau,\xi,
$
where $\eta$ and $m_0$ are given by Lemma~\ref{lem:boundary_dichotomy} for this value of $\tau$. 
Apply Lemma~\ref{lem:regularity} to $G$ with parameters $\rho,\mu$ and minimum number of clusters $m_0$. This gives a partition
$
V(G)=V_0\cup V_1\cup\dots\cup V_m
$
and a reduced graph $R$. Increasing $n_0$ if necessary, by Lemma~\ref{lem:regularity}, 
$
\delta(R)\ge \left(\frac{1}{2}-\eta\right)m.
$
Apply Lemma~\ref{lem:boundary_dichotomy} to $R$.

\smallskip
\noindent\textbf{Case (a): almost-spanning connected matching.}
Suppose $R$ contains a connected matching covering $t\ge(1-\tau)m$ vertices. By Proposition~\ref{prop:case_a},
$
\operatorname{Cyc}(G)\ge\operatorname{Cyc}(G')\ge 2^{tL-\zeta mL}.
$
Note that
$
tL-\zeta mL
\ge (1-\tau-\zeta)(n-|V_0|)
\ge (1-\xi)n.
$
We have $\operatorname{Cyc}(G)\ge 2^{(1-\xi)n}$.

\smallskip
\noindent\textbf{Case (b): two balanced components.}
Suppose $R$ has exactly two components $Y_1,Y_2$ with $|Y_i|=(1/2\pm\tau)m$ for $i=1,2$. For $i=1,2$, let
$
A_i=\bigcup_{j\in Y_i}V_j.
$
Then
$
|A_i|=(1/2\pm(\tau+\rho))n.
$
Since there are no edges in $G'$ between $A_1$ and $A_2$, every $v\in A_i$ satisfies
$
d_G(v,A_i)\ge d_{G'}(v)-|V_0|\ge d-(\mu+2\rho)n.
$
Using $d>n/2-1$ and $|A_i|\le(1/2+\tau+\rho)n$, we obtain
$
d_G(v,A_i)\ge |A_i|-\theta n
$
for some $\theta$ satisfying $\theta,\tau+\rho\ll\xi$.
Thus the partition $V(G)=V_0\cup A_1\cup A_2$ satisfies the hypotheses of Proposition~\ref{prop:case_c} with $k=2$ and $\lambda=\tau+\rho$. Therefore
$
\operatorname{Cyc}(G)\ge 2^{(1-3(\tau+\rho))n}\ge 2^{(1-\xi)n}.
$

The two cases exhaust all possibilities, completing the proof.
\end{proof}

\section{Concluding remarks}\label{sec:conclusion}

We proved that, for regular graphs of linear degree below the Dirac threshold, the minimum possible exponential rate of \(\operatorname{Cyc}(G)\) is governed by        $q=\left\lfloor\frac{n}{d+1}\right\rfloor$
and changes in discrete steps.  For \(q\ge2\), the construction \(qK_{n/q}\) shows that both the exponent \(n/q\) and the leading coefficient \(q\) in Theorem~\ref{thm:main} are asymptotically best possible. At the exact Dirac boundary, Proposition~\ref{prop:dirac_rate} gives the correct exponential rate.

A natural next problem is to sharpen the asymptotic staircase to an exact theorem at the clique points.  For fixed \(q\ge2\) and \(s\to\infty\), is \(qK_s\) the exact minimizer of \(\operatorname{Cyc}(G)\) among all \((s-1)\)-regular graphs on \(qs\) vertices?  At the Dirac boundary, it is similarly natural to ask whether \(K_{n/2,n/2}\) is the exact minimizer among \(n\)-vertex \(n/2\)-regular graphs.  

A further refinement would be to prove size-sensitive estimates, counting cyclic subsets of each fixed size.  Such local estimates could reveal more precisely how the cyclic subsets are distributed, rather than only their total number.

\section*{Acknowledgements}
HL and ZY were supported by the Institute for Basic Science (IBS-R029-C4). LW was supported by the NSFC under grant number 12471327, the National Key R\&D Program of China under grant number 2024YFA1013900, the China Scholarship Council, and the Institute for Basic Science (IBS-R029-C4).
MN was supported by the NSFC under grant number 12571381, the China Scholarship Council, and the Institute for Basic Science (IBS-R029-C4).

\medskip

\bibliographystyle{abbrv}
\addcontentsline{toc}{section}{Bibliography}
\bibliography{HC}

@article{DKM25,
  author  = {Dragani\'{c}, N. and Keevash, P. and M\"{u}yesser, A.},
  title   = {Cyclic subsets in regular {Dirac} graphs},
  journal = {International Mathematics Research Notices},
  volume  = {2025},
  number  = {14},
  pages   = {rnaf215},
  year    = {2025}
}

@article{Dirac52,
  author  = {Dirac, G. A.},
  title   = {Some theorems on abstract graphs},
  journal = {Proceedings of the London Mathematical Society},
  series  = {3},
  volume  = {2},
  pages   = {69--81},
  year    = {1952}
}

@article{KSS97,
  author  = {Koml\'{o}s, J. and S\'{a}rk\"{o}zy, G. N. and Szemer\'{e}di, E.},
  title   = {Blow-up lemma},
  journal = {Combinatorica},
  volume  = {17},
  pages   = {109--123},
  year    = {1997}
}

@article{CK09,
  author  = {Cuckler, W. and Kahn, J.},
  title   = {Hamiltonian cycles in Dirac graphs},
  journal = {Combinatorica},
  volume  = {29},
  number  = {3},
  pages   = {299--326},
  year    = {2009}
}

@article{KLSS17,
  author  = {Kim, J. and Liu, H. and Sharifzadeh, M. and Staden, K.},
  title   = {Proof of {Koml\'{o}s}'s conjecture on {H}amiltonian subsets},
  journal = {Proceedings of the London Mathematical Society},
  volume  = {115},
  number  = {5},
  pages   = {974--1013},
  year    = {2017}
}

@article {Chvatal,
    AUTHOR = {Chv\'{a}tal, V.},
     TITLE = {On {H}amilton's ideals},
   JOURNAL = {J. Combinatorial Theory Ser. B},
  FJOURNAL = {Journal of Combinatorial Theory. Series B},
    VOLUME = {12},
      YEAR = {1972},
     PAGES = {163--168},
      ISSN = {0095-8956},
   MRCLASS = {05C99},
  MRNUMBER = {294155},
MRREVIEWER = {M.\ R.\ Garey},
       DOI = {10.1016/0095-8956(72)90020-2},
       URL = {https://doi.org/10.1016/0095-8956(72)90020-2},
}

@incollection {Erdos97,
    AUTHOR = {Erd\H{o}s, P.},
     TITLE = {Some of my favorite problems and results},
 BOOKTITLE = {The mathematics of {P}aul {E}rd\H{o}s, {I}},
    SERIES = {Algorithms Combin.},
    VOLUME = {13},
     PAGES = {47--67},
 PUBLISHER = {Springer, Berlin},
      YEAR = {1997},
      ISBN = {3-540-61032-4},
   MRCLASS = {11-02 (01A60 11-03)},
  MRNUMBER = {1425174},
MRREVIEWER = {G\'{e}rald\ Tenenbaum},
       DOI = {10.1007/978-3-642-60408-9\{_}3}

@incollection{KS96,
  author    = {Koml{\'o}s, J{\'a}nos and Simonovits, Mikl{\'o}s},
  title     = {Szemer{\'e}di's Regularity Lemma and its applications in graph theory},
  booktitle = {Combinatorics, Paul Erd{\H{o}}s is Eighty, Vol. 2},
  series    = {Bolyai Society Mathematical Studies},
  volume    = {2},
  pages     = {295--352},
  publisher = {J{\'a}nos Bolyai Mathematical Society},
  address   = {Budapest},
  year      = {1996}
}

@article{CKLOT16,
  author  = {Csaba, B. and K\"{u}hn, D. and Lo, A. and Osthus, D. and Treglown, A.},
  title   = {Proof of the 1-factorization and {Hamilton} decomposition conjectures},
  journal = {Memoirs of the American Mathematical Society},
  volume  = {244},
  number  = {1154},
  pages   = {1--164},
  year    = {2016}
}

@article{KLS14,
  author  = {Krivelevich, M. and Lee, C. and Sudakov, B.},
  title   = {Robust {Hamiltonicity} of {Dirac} graphs},
  journal = {Transactions of the American Mathematical Society},
  volume  = {366},
  number  = {6},
  pages   = {3095--3130},
  year    = {2014}
}

@article{Montgomery19,
  author  = {Montgomery, R.},
  title   = {Hamiltonicity in random graphs is born resilient},
  journal = {Journal of Combinatorial Theory, Series B},
  volume  = {139},
  pages   = {316--341},
  year    = {2019}
}

@article{SSS03,
  author  = {S\'{a}rk\"{o}zy, G. N. and Selkow, S. M. and Szemer\'{e}di, E.},
  title   = {On the number of {Hamiltonian} cycles in {Dirac} graphs},
  journal = {Discrete Mathematics},
  volume  = {265},
  number  = {1--3},
  pages   = {237--250},
  year    = {2003}
}

@article{FK08,
  author  = {Frieze, A. and Krivelevich, M.},
  title   = {On two {Hamilton} cycle problems in random graphs},
  journal = {Israel Journal of Mathematics},
  volume  = {166},
  pages   = {221--234},
  year    = {2008},
  doi     = {10.1007/s11856-008-1028-8}
}

@article{SWY25,
  author  = {Sun, W. and Wei, S. and Yang, D.},
  title   = {Clique factors in random samplings of regular graphs},
  journal = {arXiv preprint arXiv:2512.20287},
  year    = {2025}
}

@article{HLMS26,
  author  = {Hunter, Z. and Liu, T. and Milojevi\'{c}, A. and Sudakov, B.},
  title   = {Cyclic subsets of tournaments},
  journal = {Random Structures \textup{\&} Algorithms},
  year    = {2026},
  doi     = {10.1002/rsa.70056}
}

@article{KOT10,
  author  = {K\"{u}hn, D. and Osthus, D. and Treglown, A.},
  title   = {Hamilton decompositions of regular tournaments},
  journal = {Proceedings of the London Mathematical Society},
  volume  = {101},
  number  = {1},
  pages   = {303--335},
  year    = {2010},
  doi     = {10.1112/plms/pdp062}
}

\end{document}